\documentclass[12pt]{amsart}

\usepackage{amssymb,verbatim,graphicx}
\usepackage[mathscr]{eucal}
\usepackage{enumerate}
\usepackage{xspace}
\usepackage[thinlines]{easymat}

\newtheorem{theorem}{Theorem}[section]
\newtheorem{lemma}[theorem]{Lemma}

\newtheorem{theoremvoid}{Theorem} 
 
\newtheorem{conjecture}[theorem]{Conjecture}

\theoremstyle{definition}
\newtheorem{example}[theorem]{Example}
\newtheorem{remark}[theorem]{Remark}

\numberwithin{equation}{section}

\def\&{\wedge}

\newcommand{\calI}{{\mathcal I}}

\newcommand{\bfu}{{\mathbf u}}
\newcommand{\bfv}{{\mathbf v}}
\newcommand{\bfx}{{\mathbf x}}
\newcommand{\e}{{\mathbf e}}

\newcommand{\R}{\mathbb{R}}

\begin{document}

\title{Beltrami fields with nonconstant proportionality factor}

\author{Jeanne N. Clelland}
\address{Department of Mathematics, 395 UCB, University of
Colorado, Boulder, CO 80309-0395}
\email{Jeanne.Clelland@colorado.edu}

\author{Taylor Klotz}
\address{Department of Mathematics, 395 UCB, University of
Colorado, Boulder, CO 80309-0395}
\email{Taylor.Klotz@colorado.edu}



\thanks{The first author was supported in part by NSF grant DMS-1206272 and a Collaboration Grant for Mathematicians from the Simons Foundation.
}

\begin{abstract}
We consider the question raised by Enciso and Peralta-Salas in \cite{EPS16}: What nonconstant functions $f$ can occur as the proportionality factor for a Beltrami field $\bfu$ on an open subset $U \subset \R^3$?  We also consider the related question: For any such $f$, how large is the space of associated Beltrami fields?  By applying Cartan's method of moving frames and the theory of exterior differential systems, we are able to improve upon the results given in \cite{EPS16}.  In particular, the answer to the second question depends crucially upon the geometry of the level surfaces of $f$. We conclude by giving a complete classification of Beltrami fields that possess either a translation symmetry or a rotation symmetry.
\end{abstract}

\maketitle

\section{Introduction}

A {\em Beltrami field} on an open set $U \subset \R^3$ is a vector field $\bfu$ on $U$ that is a solution to the PDE system
\begin{equation}
\text{curl\ } \bfu = f \bfu, \qquad \text{div\ } \bfu = 0  \label{Beltrami-def}
\end{equation}
for some smooth function $f:U \to \R$, called the {\em proportionality factor}.  When $f$ is constant, the divergence-free condition is redundant and $\bfu$ is called a {\em strong Beltrami field}.  Strong Beltrami fields are well-studied; see, e.g., \cite{EPS12}, \cite{EPS15}.

In \cite{EPS16} and \cite{EPS18}, the authors undertake a study of Beltrami fields on open subsets of $\R^3$, with a primary focus on Beltrami fields with nonconstant proportionality factor.  Both local and global issues are considered, and the most significant result is that Beltrami fields of this type are {\em rare}, in the sense that most nonconstant functions $f$ cannot occur as the proportionality factor for any nonvanishing Beltrami field, even locally.  Specifically, the following theorem is proved in \cite{EPS16}:

\begin{theoremvoid}
Let $U \subset \R^3$ be an open set, and assume that the function $f:U \to \R^3$ is nonconstant and of class $C^{6,\alpha}$.  There is a nonlinear partial differential operator $P \neq 0$, which can be computed explicitly and involves derivatives of order at most $6$, such that \eqref{Beltrami-def} has no nonzero solutions $\bfu$ unless $P[f]$ is identically zero in $U$.  In particular, \eqref{Beltrami-def} has no nonzero solutions $\bfu$ for all $f$ in an open and dense subset of $C^k(U)$ for any $k \geq 7$.
\end{theoremvoid}

This is clearly an important result; unfortunately, the operator $P$ is extremely cumbersome to compute.  Moreover, this necessary condition for $f$ is almost certainly not sufficient; the proof of the theorem shows that there is, in fact, a hierarchy of differential constraints that the function $f$ must satisfy provided that it is sufficiently smooth.  So it remains an open question precisely which proportionality factors $f$ can occur for a nonzero Beltrami field $\bfu$.

In this paper, we seek to further explore the question raised in \cite{EPS16}, namely, 

\noindent {\bf Question 1:} Which nonconstant functions $f$ can occur as proportionality factors for nonzero Beltrami fields?  

We will also consider the related question:

\noindent {\bf Question 2:} For each such function $f$, how large is the space of associated Beltrami fields?

We will approach these questions from the point of view of adapted orthonormal frame fields on $\R^3$ and exterior differential systems, and the Cartan-K\"ahler theorem (essentially a geometric version of the Cauchy-Kowalevski theorem) will play an important role in analyzing solution spaces.  (See \cite{BCG3} or \cite{CFB} for a comprehensive introduction to these topics.) Due to the limitations imposed by these tools, we will consider only locally defined, real analytic functions $f$ and Beltrami fields $\bfu$.  Therefore, topological constraints such as those discussed in \cite{EPS16} will not play any role here.  Moreover, any statement along the lines of ``assume that $X$ is nonzero" should be interpreted as ``assume that $X$ is not identically zero and restrict to the open set where $X$ is nonzero." 

Specifically, we will assume that $f$ is a nonvanishing, real analytic function on an open set $U \subset \R^3$, with $\nabla f \neq 0$ on $U$.  (Since $f$ is nonconstant, these conditions hold on a dense open subset of $U$, and we shrink $U$ accordingly if necessary.)  The Beltrami fields on $U$ with proportionality factor $f$ are the solutions $\bfu$ to the PDE system \eqref{Beltrami-def} on $U$.  We will say that a nonvanishing function $f$ {\em admits a nonzero Beltrami field} if there exists a nonzero solution $\bfu$ to the system \eqref{Beltrami-def}.

Cartan's theory has the advantage of reducing analytic questions regarding the solution space of the PDE system \eqref{Beltrami-def} to algebraic computations.  Unfortunately, some of these algebraic computations remain intractable to us, even with the assistance of computer algebra packages such as {\sc Maple}, and so we are unable to give as explicit an answer to these questions as we might like.  Nevertheless, these methods allow us to prove the following results:

\begin{theorem}\label{all-Beltrami-theorem}
The space of all Beltrami fields on $U$ is locally parametrized by 3 functions of 2 variables. 
\end{theorem}

\begin{theorem}\label{constant-Beltrami-theorem}
The space of all Beltrami fields on $U$ with constant proportionality factor is locally parametrized by 2 functions of 2 variables. 
\end{theorem}

Taken together, Theorems \ref{all-Beltrami-theorem} and \ref{constant-Beltrami-theorem} imply that a generic Beltrami field has nonconstant proportionality factor.

\begin{theorem}\label{nonconstant-Beltrami-theorem}
Let $f$ be a nonvanishing function on $U$ with $\nabla f \neq 0$ on $U$. 
\begin{itemize}
\item If the level surfaces of $f$ are totally umbilic (i.e., open subsets of planes or spheres), then $f$ admits no nonzero Beltrami fields unless the level surfaces of $f$ are contained in either parallel planes or concentric spheres, which case the solution space of \eqref{Beltrami-def} is locally parametrized by 2 functions of 1 variable.  (Note that in the case of concentric spheres, our restrictions imply that the common center of the spheres may not be contained in $U$.)
\item If the level surfaces of $f$ contain no umbilic points, then $f$ admits at most a 3-dimensional space of Beltrami fields.  \end{itemize}
\end{theorem}

In addition, we conjecture the following:

\begin{conjecture}\label{nonconstant-Beltrami-conjecture}
Nonconstant proportionality factors $f$ admitting a nonzero Beltrami field have the following properties:
\begin{itemize}
\item The space of proportionality factors $f$ admitting a nonzero Beltrami field is locally parametrized by 3 functions of 2 variables.
\item If the level surfaces of $f$ contain no umbilic points, then $f$ admits at most a 2-dimensional space of Beltrami fields. 
\item A generic proportionality factor $f$ that admits a nonzero Beltrami field admits exactly a 1-dimensional space of Beltrami fields.
\end{itemize}
\end{conjecture}

Again, we wish to emphasize that the obstacles to proving this conjecture are computational rather than theoretical in nature.  Specifically, a complete proof would require computing the real-valued solution spaces to large systems of polynomial equations, and we have so far been unable to carry these computations to completion.

This paper is organized as follows:
\begin{itemize}
\item In \S \ref{Cauchy-sec}, we apply the Cauchy-Kowalevski theorem to the PDE system \eqref{Beltrami-def} in standard coordinates to compute the size of the space of {\em all} local Beltrami fields, regardless of proportionality factor.  This computation provides a proof of Theorem \ref{all-Beltrami-theorem} and also gives some insight into the size of the space of proportionality factors $f$ that admit solutions.
\item In \S \ref{f-constant-sec}, we apply the Cauchy-Kowalevski theorem to compute the size of the space of local Beltrami fields with constant proportionality factor $f$.  This computation provides a proof of Theorem \ref{constant-Beltrami-theorem}.
\item In \S \ref{setup-sec}, we assume that $f$ is nonconstant and choose an orthonormal frame field $(\e_1, \e_2, \e_3)$ on $U$ that is adapted to the geometry of the level surfaces of $f$.  We then reformulate the PDE system \eqref{Beltrami-def} in terms of this frame field and its dual coframe field $(\omega^1, \omega^2, \omega^3)$. As a consequence, we will see that the system \eqref{Beltrami-def} may be regarded as a system of 4 equations for only 2 unknown functions on $U$ rather than the 3 component functions of $\bfu$ that appear in \eqref{Beltrami-def}.
\item In \S \ref{example-sec}, we consider some specific examples of proportionality factors $f$ and use the ideas developed in \S \ref{setup-sec} to compute the spaces of Beltrami fields that they admit.  These examples provide some intuition as to what sorts of behavior we might expect in general.
\item In \S \ref{analysis-sec}, we define an exterior differential system whose integral manifolds are in one-to-one correspondence with Beltrami fields on $U$ with proportionality factor $f$.  
We then apply Cartan's methods to analyze this exterior differential system and its integral manifolds in general.  In the course of this analysis, we will see how the geometry of the level surfaces of $f$ plays a crucial role in the existence and size of the space of integral manifolds. 
\item In \S \ref{symmetry-sec}, we consider the special case of Beltrami fields that possess either a translation symmetry or a rotation symmetry.  In both cases, the symmetry assumption simplifies the PDE system defining Beltrami fields sufficiently to allow a complete classification of local Beltrami fields with these symmetries.
\end{itemize}

\section{The space of Beltrami fields via Cauchy-Kowalevski}\label{Cauchy-sec}

In this section, we give a proof of Theorem \ref{all-Beltrami-theorem}.
We begin by writing \eqref{Beltrami-def} explicitly as a first-order system for the coordinate functions $(u^1, u^2, u^3)$ of $\bfu$:
\begin{equation}\label{Beltrami-sys-expanded}
\begin{gathered}
\frac{\partial u^2}{\partial x^3} - \frac{\partial u^3}{\partial x^2} = f u^1, \\
\frac{\partial u^3}{\partial x^1} - \frac{\partial u^1}{\partial x^3} = f u^2, \\
\frac{\partial u^1}{\partial x^2} - \frac{\partial u^2}{\partial x^1} = f u^3, \\
\frac{\partial u^1}{\partial x^1} + \frac{\partial u^2}{\partial x^2} + \frac{\partial u^3}{\partial x^3} = 0.
\end{gathered}
\end{equation}
Observe that the function $f$ can be eliminated from these equations to obtain a determined system of 3 equations for the 3 unknown functions $(u^1, u^2, u^3)$, and that this system can be put into Cauchy form:
\begin{equation}\label{Beltrami-sys-no-f}
\begin{aligned}
\frac{\partial u^1}{\partial x^3} & = \frac{\partial u^3}{\partial x^1} - \frac{u^2}{u^3} \left( \frac{\partial u^1}{\partial x^2} - \frac{\partial u^2}{\partial x^1} \right), \\
\frac{\partial u^2}{\partial x^3} & =  \frac{\partial u^3}{\partial x^2} + \frac{u^1}{u^3} \left( \frac{\partial u^1}{\partial x^2} - \frac{\partial u^2}{\partial x^1} \right), \\
\frac{\partial u^3}{\partial x^3} & = -\left( \frac{\partial u^1}{\partial x^1} + \frac{\partial u^2}{\partial x^2} \right).
\end{aligned}
\end{equation}
The initial value problem for this system is well-posed as long as we choose initial data 
\[ u^i(x^1, x^2, 0) = \bar{u}^i(x^1, x^2), \qquad 1 \leq i \leq 3 \]
along the plane $\{x^3 = 0\}$ satisfying $\bar{u}^3(x^1, x^2) \neq 0$, which, by reordering the functions $u^i$ if necessary, we may assume without loss of generality.  
By the Cauchy-Kowalevski theorem, the space of solutions---which corresponds to {\em all} Beltrami fields with all possible proportionality factors $f$---is locally parametrized by 3 functions of 2 variables.  
This completes the proof of Theorem \ref{all-Beltrami-theorem}.

Observe that, given any solution $\bfu = (u^1, u^2, u^3)$ to the system \eqref{Beltrami-sys-no-f}, we can recover the proportionality factor $f$ from the equation
\begin{equation}\label{f-from-solution}
f = \frac{1}{u^3} \left( \frac{\partial u^1}{\partial x^2} - \frac{\partial u^2}{\partial x^1} \right). 
\end{equation}
In particular, the local generality of proportionality factors $f$ admitting a nonzero Beltrami field is bounded above by the local generality of Beltrami fields---i.e., 3 functions of 2 variables.

\begin{remark}
Recall that the main result of \cite{EPS16} states that any proportionality factor $f$ admitting a nonzero Beltrami field must satisfy a certain partial differential equation of order at most 6.  The function count above shows that this necessary condition is not sufficient, as the space of solutions to a 6th order PDE is locally parametrized by 6 functions of 2 variables.
\end{remark}

Conversely, we can give a heuristic argument suggesting that the space of proportionality factors $f$ admitting a nonzero Beltrami field is locally parametrized by precisely 3 functions of 2 variables, rather than by some smaller space.  Starting with equation \eqref{f-from-solution} and making use of equations \eqref{Beltrami-sys-no-f}, we can compute expressions for $\frac{\partial f}{\partial x^3}$ and $\frac{\partial^2 f}{\partial (x^3)^2}$ in terms of the functions $u^1, u^2, u^3$ and their derivatives with respect to $x^1$ and $x^2$.  Restricting to the initial plane $\{x^3 = 0 \}$ yields a system of the form
\begin{equation}\label{system-for-initial-conditions}
\begin{aligned}
f(x^1, x^2, 0) & = \frac{1}{\bar{u}^3} \left( \frac{\partial \bar{u}^1}{\partial x^2} - \frac{\partial \bar{u}^2}{\partial x^1} \right) , \\
\frac{\partial f}{\partial x^3}(x^1, x^2, 0) & = D_1(\bar{u}^1, \bar{u}^2, \bar{u}^3), \\
\frac{\partial^2 f}{\partial (x^3)^2}(x^1, x^2, 0) & = D_2(\bar{u}^1, \bar{u}^2, \bar{u}^3), 
\end{aligned}
\end{equation}
where $D_1, D_2$ are partial differential operators involving only derivatives with respect to $x^1$ and $x^2$, applied to the initial data $\bar{u}^i$.  Now suppose that we specify initial data for $f$ and its first two $x^3$-derivatives of the form
\[ f(x^1, x^2, 0) = \bar{f}_0(x^1, x^2), \qquad \frac{\partial f}{\partial x^3}(x^1, x^2, 0) = \bar{f}_1(x^1, x^2), \qquad \frac{\partial^2 f}{\partial (x^3)^2}(x^1, x^2, 0) = \bar{f}_2(x^1, x^2), \]
where the functions $\bar{f}_0, \bar{f}_1, \bar{f}_2$ are arbitrary.  Substituting these expressions into the system \eqref{system-for-initial-conditions} yields a system of 3 PDEs for the 3 unknown functions $\bar{u}^1, \bar{u}^2, \bar{u}^3$.  Generically, we might expect this system to admit local solutions---at least in the real analytic category---for arbitrary choices of the functions $\bar{f}_i$; moreover, these solutions should be unique up to the choice of some initial data along a submanifold of strictly lower dimension (e.g., a curve in the $(x^1, x^2)$ plane).  The functions $\bar{u}^1, \bar{u}^2, \bar{u}^3$ then determine a unique solution $\bfu = (u^1, u^2, u^3)$ of the Cauchy system \eqref{Beltrami-sys-no-f}, which in turn determines a unique proportionality factor $f$ via equation \eqref{f-from-solution}. 

This argument provides evidence for the first statement in Conjecture \ref{nonconstant-Beltrami-conjecture}, strongly suggesting that the space of proportionality factors admitting a nonzero Beltrami field is locally parametrized by exactly 3 functions of 2 variables.

\section{Beltrami fields with constant proportionality factor}\label{f-constant-sec}

In this section, we give a proof of Theorem \ref{constant-Beltrami-theorem}.  To this end, suppose that $f$ is a constant function.

First, suppose that $f=0$.  In this case, it is straightforward to show that the general solution to \eqref{Beltrami-def} is
\[ \bfu = \nabla F, \]
where $F: U \to \R$ is a harmonic function on $U$; i.e., $\Delta F = 0$.  The space of harmonic functions---and hence the space of Beltrami fields with $f=0$---is locally parametrized by 2 functions of 2 variables.

Next, suppose that $f = c$ is a nonzero constant.  In this case, the first, second, and fourth equations of \eqref{Beltrami-sys-expanded} may be written as a system in Cauchy form with respect to the derivatives in the $x^3$ direction, with the third equation representing a constraint on the initial data.  A straightforward computation shows that this overdetermined system is {\em compatible}: If the initial data satisfies the constraint, then this constraint holds throughout the entire solution to the Cauchy system.  Since the constraint on the initial data is given by a first-order PDE, the space of solutions again depends locally on 2 functions of 2 variables.

For purposes of measuring the size of the solution space, this collection of ``2 functions of 2 variables for each constant $c \in \R$" still counts as 2 functions of 2 variables; the additional real constant $c$ in the initial data may be thought of as ``1 function of 0 variables," which contributes negligibly to the size of the solution space.  Thus, the space of all Beltrami fields with constant proportionality factor is locally parametrized by 2 functions of 2 variables.  This completes the proof of Theorem \ref{constant-Beltrami-theorem}.

Taken together, Theorems \ref{all-Beltrami-theorem} and \ref{constant-Beltrami-theorem} imply that the space of Beltrami fields with constant proportionality factor has ``measure zero" in the overall space of solutions to the system \eqref{Beltrami-sys-no-f}.  (This is an infinite-dimensional analog to the notion that a 2-dimensional submanifold of a 3-dimensional manifold has measure zero inside the larger manifold.)
Therefore, a generic Beltrami field must have nonconstant proportionality factor  $f$.  For the remainder of this paper, we will assume that $f$ is a given, nonconstant function and restrict to the open set $U$ where $f$ and $\nabla f$ are both nonzero.

\section{A geometric approach via adapted frame fields}\label{setup-sec}

In order to recast the system \eqref{Beltrami-def} in the language of exterior differential systems, we will make use of the canonical identifications between vector fields and differential forms in $\R^3$. Specifically, a vector field
\[ \bfv = v^1 \frac{\partial}{\partial x^1} + v^2 \frac{\partial}{\partial x^2} + v^3 \frac{\partial}{\partial x^3} \]
may be canonically identified with either the 1-form
\[ \alpha_{\bfv} = v^1\, dx^1 + v^2\, dx^2 + v^3 \, dx^3 \]
or the 2-form
\[ \beta_{\bfv} = v^1\, dx^2 \wedge dx^3 + v^2\, dx^3 \wedge dx^1 + v^3\, dx^1 \wedge dx^2. \]
Similarly, a real-valued function $f$ may either be regarded as a 0-form or canonically identified with the 3-form $\gamma_f = f\, dx^1 \wedge dx^2 \wedge dx^3$.  

Moreover, these identifications may be described more generally with respect to any orthonormal frame field: 
Let $(\e_1, \e_2, \e_3)$ be any oriented, orthonormal frame field on an open set $U \subset \R^3$, with dual coframe field $(\omega^1, \omega^2, \omega^3)$.  If $\bfv$ is a vector field on $U$ given by $\bfv = v^i \e_i$, then these identifications take the form
\begin{gather*}
\alpha_{\bfv}  = v^1\, \omega^1 + v^2\, \omega^2 + v^3\, \omega^3, \\
\beta_{\bfv}  = v^1\, \omega^2 \wedge \omega^3 + v^2\, \omega^3 \wedge \omega^1 + v^3\, \omega^1 \wedge \omega^2.
\end{gather*}
Similarly, if $f$ is a real-valued function on $U$, then
\[ \gamma_f = f\, \omega^1 \wedge \omega^2 \wedge \omega^3. \]

Under these identifications, the usual differential operators of vector calculus are all given by exterior differentiation:
\begin{itemize}
\item For any smooth function $f$, the vector field $\nabla f$ is identified with the 1-form $df$.
\item For any vector field $\bfv$ identified with the 1-form $\alpha_{\bfv}$, the vector field $\nabla \times \bfv$ is identified with the 2-form $d\alpha_{\bfv}$.
\item For any vector field $\bfv$ identified with the 2-form $\beta_{\bfv}$, the vector field $\nabla \cdot \bfv$ is identified with the 3-form $d\beta_{\bfv}$.
\end{itemize}

It follows that the system \eqref{Beltrami-def} is equivalent to the equations
\begin{equation}\label{Beltrami-eqs-form-version}
d\alpha_{\bfu} = f \beta_{\bfu}, \qquad d\beta_{\bfu} = 0,
\end{equation} 
where $\alpha_{\bfu}, \beta_{\bfu}$ represent the canonical identifications of $\bfu$ with a 1-form and 2-form, respectively.  As written, these equations are basis-independent, but they can be expressed in terms of the dual coframe field $(\omega^1, \omega^2, \omega^3)$ to any orthonormal frame field $(\e_1, \e_2, \e_3)$ on $U$.  In order to exploit this freedom, we will choose an orthonormal frame field that is nicely adapted to the geometry of the level surfaces of $f$.

The price of using such an adapted frame field is that the dual 1-forms $(\omega^1, \omega^2, \omega^3)$ are typically not exact.  Rather, there exist unique 1-forms $\{\omega^i_j = -\omega^i_j,\ 1 \leq i,j \leq 3\}$, called the {\em connection forms}, defined by the equations
\[ d\e_i = \e_j \omega^j_i, \qquad i = 1,2,3. \]
(Here and subsequently we use the Einstein summation convention, so the repeated index $j$ is summed from 1 to 3.)
The dual and connection forms satisfy the {\em Cartan structure equations}
\begin{equation}\label{structure-equations}
\begin{aligned}
d\omega^i & = -\omega^i_j \wedge \omega^j, \\
d\omega^i_j & = -\omega^i_k \wedge \omega^k_j.
\end{aligned}
\end{equation}

In order to choose our adapted frame field, first observe that equations \eqref{Beltrami-def} imply that
\[ \nabla \cdot (f\bfu) = \nabla \cdot \bfu = 0 ,\]
which in turn implies that $\nabla f \cdot \bfu = 0$, and hence that $\bfu$ is orthogonal to $\nabla f$.  So, we start by considering orthonormal frame fields $(\e_1, \e_2, \e_3)$ on $U$ with the property that
\[ \e_3 = \frac{\nabla f}{|\nabla f|}. \]
Then the vector fields $(\e_1, \e_2)$ will be tangent to the level surfaces of $f$ at each point of $U$, and the dual forms $(\omega^1, \omega^2, \omega^3)$ must satisfy
\[ \omega^3 = \frac{df}{|\nabla f|}. \]
Let $g:U \to \R$ be defined by $|\nabla f| = e^{-g}$, so that we have $\omega^3 = e^g\, df$.  Then differentiating yields 
\begin{equation} \label{d-omega3-eq}
d\omega^3 = dg \wedge \omega^3. 
\end{equation}
If we write
\[ dg = g_1 \omega^1 + g_2 \omega^2 + g_3\omega^3 \]
(i.e., $g_i$ represents the covariant derivative of $g$ with respect to the vector field $\e_i$),
then equation \eqref{d-omega3-eq} together with the Cartan structure equation
\[ d\omega^3 = -\omega^3_1 \wedge \omega^1 - \omega^3_2 \wedge \omega^2 \]
imply that
\begin{equation}\label{some-connection-forms} 
\begin{aligned}
\omega^3_1 & = h_{11} \omega^1 + h_{12} \omega^2 + g_1 \omega^3 ,\\
\omega^3_2 & = h_{12} \omega^1 + h_{22} \omega^2 + g_2 \omega^3 
\end{aligned}
\end{equation}
for some functions $h_{11}, h_{12}, h_{22}:U \to \R$.
These functions may be interpreted as follows: Because $\omega^3$ is well-defined and integrable on $U$, it defines a (local) foliation of $U$, the leaves of which are the level surfaces of $f$.  At any point $\bfx \in U$, the matrix $[h_{ij}(\bfx)]$ is the second fundamental form of the level surface $\Sigma$ of $f$ passing through the point $\bfx$ with respect to the orthonormal basis $(\e_1, \e_2)$ for the tangent plane $T_\bfx \Sigma$.
By choosing $(\e_1, \e_2)$ to be principal directions for this level surface at each point, we can arrange that this matrix is diagonal, i.e., that $h_{12} = 0$.

In order to ensure that principal vector fields $(\e_1, \e_2)$ can be chosen smoothly, we need to make a {\em constant type} assumption.  Specifically, we will assume that either:
\begin{itemize}
\item the level surfaces of $f$ contain no umbilic points in $U$, or
\item the level surfaces of $f$ in $U$ are all totally umbilic.
\end{itemize}
If neither condition holds identically on $U$, then we will restrict to the open subset of $U$ where the first condition holds.

Now, since any solution $\bfu$ to equations \eqref{Beltrami-def} has the property that $\bfu \cdot \nabla f = 0$, we can write $\bfu$ as
\[ \bfu = u^1 \e_1 + u^2 \e_2 \]
for some smooth functions $u^1, u^2:U \to \R$.  Then we have
\[ \alpha_{\bfu} = u^1\, \omega^1 + u^2\, \omega^2, \qquad \beta_{\bfu} = (-u^2\, \omega^1 + u^1\, \omega^2) \wedge \omega^3, \]
and equations \eqref{Beltrami-eqs-form-version} may be written as
\begin{equation}\label{Beltrami-form-eqs-expanded}
\begin{gathered}
d(u^1\, \omega^1 + u^2\, \omega^2) = f(-u^2\, \omega^1 + u^1\, \omega^2) \wedge \omega^3, \\ 
d((-u^2\, \omega^1 + u^1\, \omega^2) \wedge \omega^3) = 0.
\end{gathered}
\end{equation}

In order to interpret equations \eqref{Beltrami-form-eqs-expanded} as a PDE system for the functions $u^1, u^2$, define the covariant derivatives $u^i_j$ by the equations
\begin{equation}\label{introduce-u-derivs}
\begin{aligned}
du^1 & = u^1_1 \omega^1 + u^1_2 \omega^2 + u^1_3 \omega^3 , \\
du^2 & = u^2_1 \omega^1 + u^2_2 \omega^2 + u^2_3 \omega^3.
\end{aligned}
\end{equation}
Write the connection forms as
\begin{equation}\label{expanded-connection-forms}
\begin{aligned}
\omega^3_1 & = h_{11} \omega^1 + g_1 \omega^3, \\
\omega^3_2 & = h_{22} \omega^2 + g_2 \omega^3, \\
\omega^1_2 & = k_1 \omega^1 + k_2 \omega^2 + k_3 \omega^3.
\end{aligned}
\end{equation}
Applying the Cartan structure equations \eqref{structure-equations} and substituting \eqref{introduce-u-derivs} and \eqref{expanded-connection-forms} into equations \eqref{Beltrami-form-eqs-expanded} yields
\begin{equation}\label{compute-PDE-sys}
\begin{aligned}
& (u^2_1 - u^1_2 - k_1 u^1 - k_2 u^2)\, \omega^1 \wedge \omega^2 \\
& + (u^1_3  - h_{11} u^1 + (k_3 - f) u^2)\, \omega^3 \wedge \omega^1 \\
& + (-u^2_3 + (k_3 - f) u^1 + h_{22} u^2)\, \omega^2 \wedge \omega^3 = 0, \\[0.05in]
& (u^1_1 + u^2_2 + (g_1 - k_2) u^1 + (g_2 + k_1) u^2)\, \omega^1 \wedge \omega^2 \wedge \omega^3 = 0.
\end{aligned}
\end{equation}
It follows that the functions $u^1, u^2:U \to \R$ must satisfy the overdetermined PDE system
\begin{equation}\label{first-order-PDE-sys}
\begin{aligned}
u^2_1 - u^1_2 & = k_1 u^1 + k_2 u^2, \\
u^1_1 + u^2_2 & = (k_2 - g_1) u^1 - (k_1 + g_2) u^2, \\
u^1_3 & = h_{11} u^1 + (f - k_3) u^2, \\
u^2_3 & = (k_3 - f) u^1 + h_{22} u^2.
\end{aligned}
\end{equation}
Qualitatively, the first two equations say that the restrictions of the functions $u^1, u^2$ to each level surface of $f$ satisfy an elliptic PDE system, while the last two describe the evolution of the functions $u^1, u^2$ through the family of level surfaces of $f$.  
This system is compatible if and only if the evolution described by the last two equations preserves the conditions prescribed by the first two.

\section{Examples}\label{example-sec}

In this section, we consider some examples of proportionality factors $f$ for which the solution space of the system \eqref{first-order-PDE-sys} can be analyzed directly.

\begin{example}\label{f=z-ex}
Suppose that $f:U \to \R$ has the form $f(x,y,z) = \phi(z)$. (Here we assume that $\phi$ is known.)  Then we have
\[ \e_3 = \frac{\nabla f}{|\nabla f|} = \frac{\partial}{\partial z}. \]
The level surfaces of $f$ are horizontal planes, and we can take $(\e_1, \e_2, \e_3)$ to be the standard basis for $\R^3$.  Then the dual forms are given by 
\[ \omega^1 = dx, \qquad \omega^2 = dy, \qquad \omega^3 = dz. \]
It follows from the structure equations that the connection forms are given by
\[ \omega^3_1 = \omega^3_2 = \omega^1_2 = 0. \]
In this case, the PDE system \eqref{first-order-PDE-sys} becomes:
\begin{equation}\label{first-order-PDE-sys-f=z}
\begin{aligned}
u^2_x - u^1_y & = 0, \\
u^1_x + u^2_y & = 0, \\
u^1_z & = \phi(z) u^2, \\
u^2_z & = -\phi(z) u^1 .
\end{aligned}
\end{equation}
This PDE system is compatible, and the solution space can be described explicitly: The last two equations in \eqref{first-order-PDE-sys-f=z} may be regarded as an ODE system with respect to the $z$ variable.  The general solution to this subsystem is given by
\begin{equation}\label{f=z-partial-soln}
\begin{aligned}
u^1(x,y,z) & = v(x,y) \cos \Phi(z) + w(x,y) \sin \Phi(z), \\
u^2(x,y,z) & = -v(x,y) \sin \Phi(z) + w(x,y) \cos \Phi(z),
\end{aligned}
\end{equation}
where $\Phi(z)$ satisfies $\Phi'(z) = \phi(z)$.
Substituting \eqref{f=z-partial-soln} into the first two equations in \eqref{first-order-PDE-sys-f=z} yields
\begin{align*}
(w_x - v_y) \cos \Phi(z) - (v_x + w_y) \sin \Phi(z) & = 0, \\
(v_x + w_y) \cos \Phi(z) + (w_x - v_y) \sin \Phi(z) & = 0.
\end{align*}
Therefore, \eqref{f=z-partial-soln} gives a solution to the system \eqref{first-order-PDE-sys-f=z} precisely when the functions $v(x,y), w(x,y)$ satisfy the PDE system
\[ w_x - v_y = v_x + w_y = 0. \]
These are the Cauchy-Riemann equations for the pair $(w(x,y), v(x,y))$, and solutions depend locally on 2 functions of 1 variable.  
It follows that the space of Beltrami fields with proportionality factor $f = \phi(z)$ depend locally on 2 functions of 1 variable. 

\end{example}

\begin{example}\label{cylinders-ex}
Suppose that $f:U \to \R$ has the form $f(x, y, z) = \phi(\sqrt{x^2 + y^2})$, or, in cylindrical coordinates, $f(r, \theta, z) = \phi(r)$.  (Again, we assume that $\phi$ is known.)  Then we have
\[ \e_3 = \frac{1}{\sqrt{x^2 + y^2}} \left( x \frac{\partial}{\partial x} + y \frac{\partial}{\partial y} \right) = \frac{\partial}{\partial r}. \]
The level surfaces of $f$ are concentric circular cylinders, and a principal orthonormal frame field is given by
\[
\e_1 = \frac{1}{\sqrt{x^2 + y^2}} \left( -y \frac{\partial}{\partial x} + x \frac{\partial}{\partial y} \right) = \frac{1}{r} \frac{\partial}{\partial \theta}, \qquad
\e_2 = \frac{\partial}{\partial z}.
\]
Then the dual forms are given by
\[ \omega^1 = r\, d\theta, \qquad \omega^2 = dz, \qquad  \omega^3 = dr.\]
It follows from the structure equations that the connection forms are given by
\[ \omega^1_2 = \omega^3_2 = 0, \qquad \omega^3_1 = -d\theta. \]
In this case, the PDE system \eqref{first-order-PDE-sys} becomes:
\begin{equation}\label{first-order-PDE-sys-f=r}
\begin{aligned}
u^2_{\theta} - r u^1_z & = 0, \\
u^1_{\theta} + r u^2_z & = 0, \\
u^1_r & = - \frac{1}{r} u^1 + \phi(r) u^2, \\
u^2_r & = -\phi(r) u^1 .
\end{aligned}
\end{equation}
In order to explore the compatibility of the system \eqref{first-order-PDE-sys-f=r}, differentiate the 2nd equation with respect to $r$, the 3rd equation with respect to $\theta$, and the 4th equation with respect to $z$.  This process yields
\begin{align*}
0 & = u^1_{r\theta} + u^2_z + r u^2_{rz} \\
& = \left( -\frac{1}{r} u^1_{\theta} + \phi(r) u^2_{\theta}\right) + u^2_z - r \phi(r) u^1_z \\
& = -\frac{1}{r} u^1_{\theta} + u^2_z.
\end{align*} 
Together with the 2nd equation, this implies that $u^1_{\theta} = u^2_z = 0$.
Therefore, we have
\[ u^1 = u^1(r, z), \qquad u^2 = u^2(r, \theta). \]
Now consider the 3rd equation in \eqref{first-order-PDE-sys-f=r}: 
\[ u^1_r = - \frac{1}{r} u^1 + \phi(r) u^2. \]
Differentiating with respect to $\theta$ yields $\phi(r) u^2_{\theta} =0$, and hence $u^2_{\theta} = 0$.  
Similarly, differentiating the 4th equation in \eqref{first-order-PDE-sys} with respect to $z$ shows that $u^1_z = 0$.
It follows that $u^1$ and $u^2$ are functions of $r$ alone---and hence constant on each level surface of $f$---and that they satisfy the ODE system given by the 3rd and 4th equations in \eqref{first-order-PDE-sys-f=r}.
Therefore, the solution space of the system \eqref{first-order-PDE-sys-f=r} is a 2-dimensional space, parametrized by arbitrary initial values $(u^1(r_0), u^2(r_0))$ for $(u^1(r), u^2(r))$ on some cylinder $r = r_0$.

\end{example}

\begin{example}
Suppose that $f:U \to \R$ has the form $f(x, y, z) = \phi(\tan^{-1}(y/x))$, or, in cylindrical coordinates, $f(r, \theta, z) = \phi(\theta)$.  The level surfaces of $f$ are open subsets of planes containing the $z$-axis (although our hypotheses exclude points of the $z$-axis from $U$), and
we can use the same orthonormal frame field as in Example \ref{cylinders-ex}, but in a different order:
\[ \e_1 = \frac{\partial}{\partial z}, \qquad \e_2 = \frac{\partial}{\partial r}, \qquad \e_3 = \frac{1}{r} \frac{\partial}{\partial \theta}. \]
Then the dual forms are given by
\[ \omega^1 = dz, \qquad \omega^2 = dr, \qquad  \omega^3 = r\, d\theta.\]
It follows from the structure equations that the connection forms are given by
\[ \omega^3_1 = \omega^1_2 = 0, \qquad \omega^3_2 = d\theta. \]
In this case, the PDE system \eqref{first-order-PDE-sys} becomes:
\begin{equation}\label{first-order-PDE-sys-f=theta}
\begin{aligned}
u^2_z - u^1_r & = 0, \\
u^1_z + u^2_r & = -\frac{1}{r} u^2, \\
u^1_{\theta} & = r \phi(\theta) u^2, \\
u^2_{\theta} & = -r \phi(\theta) u^1 .
\end{aligned}
\end{equation}
In order to explore the compatibility of the system \eqref{first-order-PDE-sys-f=theta}, differentiate the 2nd equation with respect to $\theta$, the 3rd equation with respect to $z$, and the 4th equation with respect to $r$.  This process yields
\begin{align*}
0 & = u^1_{z\theta} + u^2_{r\theta} + \frac{1}{r} u^2_{\theta} \\
& =  r \phi(\theta) u^2_z + (-\phi(\theta) u^1 - r\phi(\theta) u^1_r) + \frac{1}{r}(-r \phi(\theta) u^1) \\
& = \phi(\theta) \left( r( u^2_z - u^1_r) - 2r u^1 \right).
\end{align*} 
Taking the 1st equation in \eqref{first-order-PDE-sys-f=theta} into account, this becomes
\[ 0 = -2r \phi(\theta) u^1; \]
therefore, we must have $u^1 = 0$.  But then the 3rd equation in \eqref{first-order-PDE-sys-f=theta} implies that $u^2=0$ as well.  Hence there are no nonzero Beltrami fields with proportionality factor $f = \phi(\theta)$.

\end{example}

\section{Exterior differential system analysis}\label{analysis-sec}
The PDE system \eqref{first-order-PDE-sys} may be reformulated as an exterior differential system as follows.  Observe that, algebraically, we may solve the first two equations in \eqref{first-order-PDE-sys} by setting
\begin{equation}\label{solve-for-some-derivatives}
\begin{aligned}
u^1_1 & = p_1 + (k_2 - g_1)u^1, \\
u^1_2 & = p_2 - k_2 u^2, \\
u^2_1 & = p_2 + k_1 u^1, \\
u^2_2 & = -p_1 - (k_1 + g_2) u^2
\end{aligned} 
\end{equation}
for some arbitrary functions $p_1, p_2:U \to \R$.  So, let $M = U \times \R^4$, with coordinates $(u^1, u^2, p_1, p_2)$ on the $\R^4$ factor, and
let $\calI$ be the differential ideal on $M$ generated by the 1-forms
\begin{equation}
\begin{aligned}
\theta^1 & = du^1 - (p_1 + (k_2 - g_1)u^1)\,\omega^1 - (p_2 - k_2 u^2)\,\omega^2 - (h_{11} u^1 + (f - k_3) u^2)\, \omega^3, \\
\theta^2 & = du^2 - (p_2 + k_1 u^1)\,\omega^1 - (-p_1 - (k_1 + g_2) u^2)\,\omega^2 - ((k_3 - f) u^1 + h_{22} u^2)\,\omega^3,
\end{aligned} \label{define-I}
\end{equation} 
and their exterior derivatives.  An {\em integral manifold} of $(M, \calI)$ is a submanifold $\iota:N \hookrightarrow M$ with the property that $\iota^*(\calI) = (0)$.
Three-dimensional integral manifolds of $(M, \calI)$ satisfying the independence condition $\iota^*(\omega^1 \wedge \omega^2 \wedge \omega^3) \neq 0$ are in one-to-one correspondence with solutions of the PDE system \eqref{first-order-PDE-sys} on $U$, and hence with Beltrami fields on $U$ with proportionality factor $f$.

Cartan's algorithm for computing the space of integral manifolds of the exterior differential system $(M, \calI)$ with independence condition $\Omega = \omega^1 \wedge \omega^2 \wedge \omega^3$ begins by computing the 2-forms $d\theta^1, d\theta^2$ modulo the {\em algebraic} ideal generated by $\{\theta^1, \theta^2\}$.  This yields expressions of the form 
\begin{equation}\label{dthetas-first-pass}
\begin{aligned}
d\theta^1 & \equiv  -dp_1 \wedge \omega^1 - dp_2 \wedge \omega^2 + T^1_{ij} \omega^i \wedge \omega^j , \\
d\theta^2 & \equiv -dp_2 \wedge \omega^1 + dp_1 \wedge \omega^2 + T^2_{ij} \omega^i \wedge \omega^j,
\end{aligned}
\end{equation}
where the $T^k_{ij}$ are functions on $M$ involving the known functions $f, g, h_{ij}, k_i$ on $U$ and their derivatives, as well as the unknowns $u^1, u^2, p_1, p_2$ on $M$.  The functions $T^k_{ij}$ are called {\em torsion} functions for the exterior differential system $(M, \calI)$.

The next step is to determine whether there exist functions $p_{ij}$ on $M$ such that the 1-forms $\pi_1, \pi_2$ defined by
\[ \pi_i = dp_i - p_{ij} \omega^j \]
satisfy
\begin{equation}\label{dthetas-first-pass-absorbed}
\begin{aligned}
d\theta^1 & \equiv  -\pi_1 \wedge \omega^1 - \pi_2 \wedge \omega^2 , \\
d\theta^2 & \equiv -\pi_2 \wedge \omega^1 + \pi_1 \wedge \omega^2.
\end{aligned}
\end{equation}
(Note that these are affine linear equations for the functions $p_{ij}$, so the existence of such functions can be determined via linear algebra.)
If no such functions $p_{ij}$ exist, then the system $(M, \calI, \Omega)$ has no 3-dimensional integral manifolds, and hence there are no Beltrami fields with with proportionality factor $f$.  (In more common PDE terminology, the nonexistence of such functions means that imposing the condition that mixed partial derivatives commute produces a contradiction.)
If, on the other hand, such functions $p_{ij}$ do exist, we say that {\em the torsion can be absorbed}.

A straightforward\footnote{While straightforward in principle, most of the computations in this algorithm are impractical to carry out by hand.  We have used {\sc Maple} for all computations, along with the {\sc Cartan} package developed by the first author and available for free download at http://euclid.colorado.edu/$\tilde{\ }$jnc/Maple.html.}  computation shows that the torsion can be absorbed if and only if 
\begin{equation}\label{torsion-first-pass}
\begin{aligned}
2(& h_{11} - h_{22}) p_1 \\
& + (g_{13} + (h_{11})_1 -(h_{22})_1 - g_1 h_{22} + g_2(k_3 - 2f)) u^1 \\
& \ \ \  + (g_{23} + (h_{22})_2 - (h_{11})_2 - g_1(k_3 - 2f) - g_2 h_{11}) u^2 \\
& \ \ \ \ \  = 0,
\end{aligned}
\end{equation}
where subscripts indicate covariant derivatives with respect to the vector fields $\e_i$.
This equation should be regarded as a linear equation for the unknown quantities $p_1, u^1, u^2$ on $M$, with coefficients depending on known quantities on $U$.  

At this point, there are three possibilities to consider:
\begin{enumerate}
\item Equation \eqref{torsion-first-pass} holds identically on $M$.  This happens if and only if the coefficients of $p_1, u^1$, and $u^2$ each vanish identically on $U$.
\item The coefficient of $p_1$ vanishes identically on $U$, but at least one of the other two coefficients does not.
\item The coefficient of $p_1$ is nonzero on $U$.
\end{enumerate}
We will consider each of these cases in separate subsections below.

\subsection{Case 1: Equation \eqref{torsion-first-pass} holds identically on $M$}

This condition means that the coefficients of $p_1, u^1$, and $u^2$ in \eqref{torsion-first-pass} must all vanish identically on $U$. 
From the vanishing of the coefficient of $p_1$, we have 
\begin{equation}\label{tot-umbilic-cond}
 h_{11} - h_{22} = 0. 
\end{equation}
This means that all the level surfaces of $f$ are totally umbilic (i.e., open subsets of either planes or spheres), and therefore the function $h = h_{11} = h_{22}$, which represents the normal curvature in each direction along the level surfaces of $f$, is constant on each level surface of $f$.
Taking this into account, the vanishing of the coefficients of $u^1$ and $u^2$ reduces to
\begin{equation}\label{umbilic-vanishing-torsion}
 g_{13} - g_1 h +  g_2(k_3 - 2f) = g_{23}  - g_1(k_3 - 2f) - g_2 h = 0. 
\end{equation}

\begin{lemma}\label{tot-umbilic-vanishing-torsion-lemma}
Equations \eqref{tot-umbilic-cond} and  \eqref{umbilic-vanishing-torsion} imply that $g_1 = g_2 = 0$.
\end{lemma}

\begin{proof}
Suppose not.  Since the level surfaces of $f$ are totally umbilic, we can rotate the tangent frame vectors $(\e_1, \e_2)$ at each point however we like without changing the condition that they form a principal adapted frame field for the level surfaces of $f$.  Under such a rotation, the vector $[ g_1 \ g_2 ]$ is rotated through the same angle as the tangent frame vectors at each point.  Thus we can choose an adapted frame field for which $g_2 = 0$ and $g_1 < 0$.
Under this assumption, the connection forms \eqref{expanded-connection-forms} may be written as
\begin{equation}\label{umbilic-expanded-connection-forms}
\begin{aligned}
\omega^3_1 & = h \omega^1 + g_1 \omega^3, \\
\omega^3_2 & = h \omega^2 , \\
\omega^1_2 & = k_1 \omega^1 + k_2 \omega^2 + k_3 \omega^3,
\end{aligned}
\end{equation}
and equations \eqref{umbilic-vanishing-torsion} become
\begin{equation*}
 g_{13} - g_1 h = -g_1(k_3 - 2f)  = 0. 
\end{equation*}
Since $g_1 \neq 0$, it follows that 
\begin{equation}\label{umbilic-vanishing-torsion-1}
 g_{13} = g_1 h, \qquad k_3 = 2f. 
\end{equation}

\begin{remark}\label{g1=curvature-remark}
Some of the functions appearing in the connection forms \eqref{umbilic-expanded-connection-forms} may be interpreted geometrically.  (For instance, we have already seen that $h$ represents the normal curvature in each direction along the level surfaces of $f$.)
 Let $\alpha:I \to U$ be any integral curve of the vector field $\e_3$. Then $\alpha$ is a unit speed curve with $\alpha'(s) = \e_3(\alpha(s))$.  Its acceleration vector at the point $\alpha(s)$ is given by
\[ \alpha''(s) = \frac{d}{ds} \e_3(\alpha(s)) = d\e_3(\alpha'(s)) = d\e_3(\e_3) = -\e_1 \omega^3_1(\e_3) = -g_1 \e_1. \]
Therefore, $\e_1$ is the Frenet normal vector to $\alpha$, and the curvature of $\alpha$ is $\kappa = -g_1$.  Similar reasoning applied to the Frenet binormal vector $\e_2$ shows that $\alpha$ has torsion $\tau = -k_3$.
\end{remark}

Now consider the structure equations \eqref{structure-equations}.  Substituting \eqref{umbilic-expanded-connection-forms} into the structure equation 
\[ d\omega^3_2 = -\omega^3_1 \wedge \omega^1_2 \]
yields
\[ -(g_1 k_1 \omega^1 + (h_3 + g_1 k_2 - h^2) \omega^2) \wedge \omega^3 = 0. \]
Since $g_1 \neq 0$, it follows that 
\[ k_1 = 0, \qquad h_3 = h^2 - g_1 k_2. \]
Next, substituting \eqref{umbilic-expanded-connection-forms} into the structure equation
\[ d\omega^3_1 = -\omega^3_2 \wedge \omega^2_1 = \omega^3_2 \wedge \omega^1_2 \]
and taking \eqref{umbilic-vanishing-torsion-1} into account yields
\[ ((g_{11} + g_1^2 + g_1 k^2) \omega^1 + g_{12} \omega^2) \wedge \omega^3 = 0, \]
and hence
\[ g_{11} = -(g_1^2 + g_1 k_2), \qquad g_{12} = 0. \]
Now, substituting \eqref{umbilic-expanded-connection-forms} into the structure equation
\[ d\omega^1_2 = -\omega^1_3 \wedge \omega^3_2 = \omega^3_1 \wedge \omega^3_2 \]
and reducing modulo $\omega^2$ yields
\[ 2f (g_1 - k_2) \omega^1 \wedge \omega^3 \equiv 0 \mod{\omega^2}, \]
while computing $d(dg_1) \equiv 0$ modulo $\omega^1$ yields
\[ -2f g_1 (g_1 + k_2) \omega^2 \wedge \omega^3 \equiv 0 \mod{\omega^1}. \]
But since we have assumed that $f$ and $g_1$ are both nonzero, this implies that $g_1 = k_2 = 0$, which is a contradiction.  Therefore, we must have $g_1 = g_2 = 0$, as claimed.

\end{proof}

Now, let $\alpha:I \to U$ be any integral curve of the vector field $\e_3$, as in Remark \ref{g1=curvature-remark}.  The condition $g_1 = g_2 = 0$ implies that $\omega^3_1(\e_3) = \omega^3_2(\e_3) = 0$, which in turn implies that
\[ \alpha''(s) = d\e_3(\e_3) = -\e_1 \omega^3_1(\e_3) - \e_2 \omega^3_2(\e_3) = 0. \]
Therefore, $\alpha$ is a straight line in $\R^3$.

To summarize, we have now shown that if equation \eqref{torsion-first-pass} holds identically on $M$, then:
\begin{enumerate}
\item Every level surface of $f$ is totally umbilic. 
\item All integral curves of the vector field consisting of the unit normal vectors to the level surfaces of $f$ are straight lines.
\end{enumerate}
Conversely, these conditions on the level surfaces of $f$ suffice to guarantee that \eqref{torsion-first-pass} holds identically on $M$.
Together, these conditions are equivalent to the condition that the level surfaces of $f$ are contained either in parallel planes (if $h=0$) or concentric spheres (if $h \neq 0$).  Note that in the latter case, our assumptions on the nonvanishing of $f$ and $\nabla f$ imply that the common center of these spheres is not contained in $U$.

Since \eqref{torsion-first-pass} holds identically on $M$, the torsion can be absorbed; thus there exist 1-forms $\pi_1, \pi_2$ on $M$ such that equations \eqref{dthetas-first-pass-absorbed} hold.  We can write these equations in the form
\begin{equation}\label{totally-umbilic-tableau}
\begin{bmatrix} d\theta^1 \\[0.05in] d\theta^2 \end{bmatrix} \equiv 
- \begin{bmatrix} \pi_1 & \pi_2 \\[0.05in] \pi_2 & -\pi_1 \end{bmatrix} 
\wedge \begin{bmatrix} \omega^1 \\[0.05in] \omega^2 \end{bmatrix} \mod{ \{\theta^1, \theta^2\} }. 
\end{equation}
In Cartan's algorithm, a straightforward computation shows that the matrix on the right-hand side of equation \eqref{totally-umbilic-tableau} is an {\em involutive tableau} with Cartan characters $s_1 = 2$, $s_2=0$.  It follows from the Cartan-K\"ahler theorem that the space of integral manifolds of the system $(M, \calI)$ with independence condition $\Omega \neq 0$ is locally parametrized by 2 functions of 1 variable.  (Note that the case $h=0$ is precisely the scenario of Example \ref{f=z-ex}.)

\begin{remark}
This case is the only time that we will need the full strength of the Cartan-K\"ahler theorem, and hence the only time when we truly need to assume that $f$ is real analytic.  In all other cases, existence results will follow from the Frobenius theorem, while nonexistence results require only sufficient smoothness to compute enough derivatives to derive the necessary contradictions.  
And in fact, as we saw in Example \ref{f=z-ex}, even in this case real analyticity of $f$ may not be strictly necessary: The Cartan-K\"ahler theorem is really only needed to prove existence of local restricted solutions on each level surface of $f$, which are open subsets of planes or spheres.  The fact that these restricted solutions may be combined to produce a consistent solution on an open subset of $U$ may then be proved by ODE techniques, which only require finite 
differentiability of $f$.
\end{remark}

\subsection{Case 2: The coefficient of $p_1$ in \eqref{torsion-first-pass} vanishes identically on $U$, but at least one of the other two coefficients does not}

As in the previous case, the vanishing of the coefficient of $p_1$ implies that $h_{11} = h_{22} = h$, and hence all the level surfaces of $f$ are totally umbilic.  The nonvanishing of at least one of the other coefficients implies that at least one of $g_1, g_2$ is nonzero; thus, as in the proof of Lemma \ref{tot-umbilic-vanishing-torsion-lemma}, we may choose an adapted frame field $(\e_1, \e_2, \e_3)$ on $U$ for which $g_2=0$ and $g_1 < 0$.  Then the connection forms \eqref{expanded-connection-forms} may be written as in \eqref{umbilic-expanded-connection-forms}.  We will consider separately the cases where $h=0$ and $h \neq 0$.

\subsubsection{Case 2.1: $h=0$.}
In this case, the torsion condition \eqref{torsion-first-pass} reduces to
\begin{equation}\label{torsion-umbilic=h=0}
g_{13}\, u^1 - g_1(k_3 - 2f) u^2  = 0.
\end{equation}

Let $\alpha:I \to U$ be any integral curve of the vector field $\e_3$.  For ease of notation, let $f(s)$ and $\e_i(s)$ denote $f(\alpha(s))$ and $\e_i(\alpha(s))$, respectively.
As mentioned in Remark \ref{g1=curvature-remark}, we have $\alpha'(s) = \e_3(s)$, and the $\omega^3$ terms in the connection forms \eqref{umbilic-expanded-connection-forms} imply that 
\begin{equation}\label{case-2-1-Frenet-eqs}
 \e_3'(s) = -g_1\e_1(s), \qquad \e_1'(s) = g_1 \e_3(s) - k_3 \e_2(s), \qquad \e_2'(2) = k_3 \e_1(s). 
\end{equation}
Therefore, the Frenet frame of $\alpha$ is given by $(\e_3, \e_1, \e_2)$, and the curvature and torsion of $\alpha$ are given by
\[ \kappa(s) = -g_1(\alpha(s)), \qquad \tau(s) = -k_3(\alpha(s)). \]
Since $g_1 < 0$, we have $\kappa(s) > 0$; in particular, $\alpha$ is not a straight line.

The level plane of $f$ passing through the point $\alpha(s)$ is spanned by the vectors $\e_1(s), \e_2(s)$; note that, by definition, $f = f(s)$ at every point of this plane.  Moreover, the adapted frame field $(\e_1, \e_2, \e_3)$ on $U$ is identical at each point of this plane.
This suggests a geometrically natural local coordinate system $(s, v, w)$ on $U$ defined by
\[ \bfx(s,v,w) = \alpha(s) + v \e_1(s) + w \e_2(s). \]
Then we have
\[ d\bfx = \e_1(s) (dv - w \tau(s)\, ds) + 
\e_2(s) ( dw + v \tau(s)\, d(s))
+ \e_3(s) (1 - v \kappa(s))\, ds,
 \]
and so the dual forms are given by
\begin{equation}\label{case-2-1-dual-forms}
\omega^1 = dv - w \tau(s)\, ds, \qquad \omega^2 = dw + v \tau(s)\, d(s), \qquad \omega^3 = (1 - v \kappa(s))\, ds.
\end{equation}
The Cartan structure equations \eqref{structure-equations} may then be used to show that the connection forms are given by
\[ \omega^3_1 = -\frac{\kappa(s)}{1 - v\kappa(s)}\, \omega^3, \qquad \omega^3_2 = 0, \qquad \omega^1_2 = -\frac{\tau(s)}{1 - v\kappa(s)} \,\omega^3. \]
Thus we have
\[ g_1 = -\frac{\kappa(s)}{1 - v\kappa(s)}, \qquad k_3 = -\frac{\tau(s)}{1 - v\kappa(s)}, \] 
and, since $g_{13}$ is defined by the covariant equation $dg_1 = g_{ij} \omega^j$, it is straightforward to compute that
\[ g_{13} = -\frac{\kappa'(s) + w \kappa(s)^2 \tau(s)}{(1 - v\kappa(s))^3}. \]

The torsion absorption condition \eqref{torsion-umbilic=h=0} may now be written (after clearing denominators) as
\begin{equation}\label{torsion-umbilic=h=0-coords}
A u^1 +B u^2 = 0,
\end{equation}
where
\begin{equation}\label{case-2-1-AB}
 A = \kappa'(s) + w \kappa(s)^2 \tau(s), \qquad B = \kappa(s)(1 - v\kappa(s))(\tau(s) + 2f(s)(1 - v\kappa(s))). 
\end{equation}
By hypothesis, at least one of the coefficients of $u^1, u^2$ in \eqref{torsion-umbilic=h=0-coords} is nonvanishing, so this equation defines a relation between the unknown functions $u^1, u^2$ that must hold on any integral manifold.  

Next, we compute the covariant derivative of equation \eqref{torsion-umbilic=h=0-coords} in the direction of $\e_3$.  From equations \eqref{first-order-PDE-sys}, in this case we have
\[ u^1_3 = (f(s) + \tau(s)) u^2, \qquad u^2_3 = -(f(s) + \tau(s)) u^1, \]
and so we obtain
\begin{equation}\label{torsion-deriv-umbilic=h=0-coords}
(A_3 -(f(s) + \tau(s)) B)u^1 + (B_3 + (f(s) + \tau(s))A)u^2 = 0.
\end{equation}
In order for equations \eqref{torsion-umbilic=h=0-coords} and \eqref{torsion-deriv-umbilic=h=0-coords} to admit a nonzero solution $(u^1, u^2)$, we must have
\begin{equation}\label{case-2-1-this-must-vanish}
A(B_3 + (f(s) + \tau(s))A) - B(A_3 -(f(s) + \tau(s)) B) = 0.
\end{equation}
From equations \eqref{case-2-1-dual-forms}, we can compute that for any function $F(s,v,w)$, the covariant derivative $F_3$ is given by
\[ F_3 = \frac{1}{1 - v \kappa(s)}\left(\frac{\partial F}{\partial s} + \tau(s) \left(w \frac{\partial F}{\partial v} - v \frac{\partial F}{\partial w} \right) \right). \]
Applying these formulas to compute $A_3$, $B_3$ and substituting these expressions (together with \eqref{case-2-1-AB}) into equation \eqref{case-2-1-this-must-vanish} yields a polynomial in the variables $(v,w)$ whose coefficients are functions of $s$ that must all vanish identically.  In particular, the highest-order term of this polynomial is
\[ 4 f(s)^3 \kappa(s)^7 v^5. \]
Since $f(s) \neq 0$, we must have $\kappa(s) = 0$.  But this contradicts the assumption that $\alpha$ is not a straight line; therefore, there are no nonzero Beltrami fields in this case.

\subsubsection{Case 2.2: $h \neq 0$.}
In this case, the torsion condition \eqref{torsion-first-pass} reduces to
\begin{equation}\label{torsion-umbilic-h-nonzero}
(g_{13} - g_1 h) u^1 - g_1(k_3 - 2f) u^2  = 0.
\end{equation}

Since the level spheres of $f$ are not concentric in this case, their centers form a regular curve in $\R^3$.  (Note that this curve is generally not contained in $U$.)  So, let $\alpha:I \to \R^3$ be the curve consisting of the centers of the level spheres of $f$, parametrized by arc length.  Let $(T(s), N(s), B(s))$ denote the Frenet frame of $\alpha$ and $\kappa(s), \tau(s)$ the curvature and torsion of $\alpha$.  (We allow for the possibility that $\alpha$ is a straight line, in which case we take $\kappa(s) = \tau(s) = 0$ and let $(T(s), N(s), B(s))$ be any smooth orthonormal frame field along $\alpha$ with $\alpha'(s) = T(s)$.)

Let $\rho(s)$ denote the radius of the level sphere of $f$ with center $\alpha(s)$, and let $f(s)$ denote the value of $f$ on this sphere.  We can define a geometrically natural local coordinate system $(s,v,w)$ on $U$ by
\[ \bfx(s,v,w) = \alpha(s) + \rho(s) \left[ (\cos v) T(s) + (\sin v)((\cos w) N(s) + (\sin w) B(s)) \right]. \]
An adapted frame field on $U$ may then be defined by
\[ \e_1(s,v,w) = \frac{\bfx_v}{|\bfx_v|}, \qquad \e_2(s,v,w) = \frac{\bfx_w}{|\bfx_w|}, \qquad \e_3(s,v,w) = \e_1(s,v,w) \times \e_2(s,v,w). \]

The argument proceeds in a similar fashion to that in the previous case, although the computations are messier.  Expressing the differential $d\bfx$ in terms of the basis $(\e_1, \e_2, \e_3)$ shows that the dual forms are given by
\begin{equation}\label{case-2-2-dual-forms}
\begin{aligned}
\omega^1 & = \rho(s)\,dv +(\kappa(s) \rho(s) \cos w - \sin v) \, ds, \\
\omega^2 & = \rho(s) (\sin v)\, dw + \rho(s) (\tau(s) \sin v - \kappa(s) \cos v \sin w)  \, d(s), \\
\omega^3 & = (\rho'(s) + \cos v)\, ds.
\end{aligned}
\end{equation}
The Cartan structure equations \eqref{structure-equations} may then be used to show that the connection forms are given by
\begin{align*} 
\omega^3_1 & = -\frac{1}{\rho(s)} \left( \omega^1 + \frac{\sin v}{(\rho'(s) + \cos v)} \omega^3 \right), \\
\omega^3_2 & = -\frac{1}{\rho(s)} \omega^2,  \\
\omega^1_2 & = -\frac{\cos v}{\rho(s) \sin v} \omega^2 - \frac{\kappa(s) \sin w}{\sin v\, (\rho'(s) + \cos v)}\omega^3. 
\end{align*}
Thus we have 
\[ g_1 = -\frac{\sin v}{\rho(s) (\rho'(s) + \cos v)}, \qquad k_3 = -\frac{\kappa(s) \sin w}{\sin v\, (\rho'(s) + \cos v)}, \qquad h = \frac{1}{\rho(s)}, \]
and computing the covariant derivative of $g_1$ in the $\e_3$ direction  shows that 
\[ g_{13} = \frac{(\rho(s) \rho''(s) + \rho'(s)^2 - 1) \sin v + \rho(s) \kappa(s) (\rho'(s) \cos v + 1)\cos w }{\rho(s)^2(\rho'(s) + \cos v)^3}. \]

The torsion absorption condition \eqref{torsion-umbilic-h-nonzero} may now be written (after clearing denominators) as
\begin{equation}\label{torsion-umbilic-h-nonzero-coords}
A u^1 +B u^2 = 0,
\end{equation}
where
\begin{equation}\label{case-2-2-AB}
\begin{aligned}
 A & = (\rho(s) \rho''(s) - 2\rho'(s)\cos v - \cos^2 v - 1) \sin v + \rho(s) \kappa(s) (\rho'(s) \cos v + 1)\cos w,
 \\
 B & = -\rho(s) (\rho'(s) + \cos v)(2 f(s) (\rho'(s) + \cos v)\sin v + \kappa(s) \sin w) . 
\end{aligned}
\end{equation}
As in the previous case, computing the covariant derivative of equation \eqref{torsion-umbilic-h-nonzero-coords} in the direction of $\e_3$ yields a second linear relation between $u^1$ and $u^2$.  The analogous necessary condition for the existence of a common nonzero solution $(u^1, u^2)$ to this equation and \eqref{torsion-umbilic-h-nonzero-coords} is a large trigonometric polynomial in the variables $(v,w)$ whose coefficients are functions of $s$ that must all vanish identically.  In particular, the highest-order term of this polynomial is
\[ \rho(s) f(s) (1 + 4 \rho(s)^2 f(s)^2)\sin v \cos^8 v. \] 
Since $f(s)$ and $\rho(s)$ are nonzero, this term is never zero; therefore, there are no nonzero Beltrami fields in this case.

\subsection{Case 3: The coefficient of $p_1$ in \eqref{torsion-first-pass} is nonzero on $U$}
This condition means that $h_{11} \neq h_{22}$ on $U$, and so the level sets of $f$ have no umbilic points.  In this case, we can choose an adapted frame field on $U$ as follows: let $\e_3 = \frac{\nabla f}{|\nabla f|}$ as before, and at each point $\bfx \in U$, let $\e_1, \e_2$ be tangent to the principal directions of the level surface $\Sigma$ of $f$ passing through $\bfx$.  

It turns out to be convenient to set
\[ h_{11} = h_1 + h_2, \qquad h_{22} = h_1 - h_2, \]
with $h_2 \neq 0$.  Then the connection forms \eqref{expanded-connection-forms} may be written as
\begin{equation}\label{nonumbilic-MC-forms}
\begin{aligned}
\omega^3_1 & = (h_1 + h_2) \omega^1 + g_1 \omega^3, \\
\omega^3_2 & = (h_1 - h_2) \omega^2 + g_2 \omega^3, \\
\omega^1_2 & = k_1 \omega^1 + k_2 \omega^2 + k_3 \omega^3.
\end{aligned}
\end{equation}

Equation \eqref{torsion-first-pass} may be interpreted as defining a codimension 1 submanifold $M' \subset M$, diffeomorphic to $U \times \R^3$ (with coordinates $(u^1, u^2, p_2)$ on the $\R^3$ factor), with the property that any integral manifold of $(M, \calI)$ must be contained in $M'$. 
Thus we must replace $\calI$ with its pullback $\calI'$ to $M'$ and consider the system $(M', \calI')$.  Algebraically, this simply means that we solve equation \eqref{torsion-first-pass} for $p_1$ and substitute the resulting expression into the differential forms in $\calI$ to obtain the system $\calI'$.
This system is still generated by $\theta_1, \theta_2$ as in \eqref{define-I}, with $p_1$ determined by equation \eqref{torsion-first-pass}.  Modulo $\{\theta_1, \theta_2\}$, we now have
\begin{equation}\label{nonumbilic-dthetas-2}
\begin{aligned}
d\theta^1 & \equiv  - dp_2 \wedge \omega^2 + T^1_{ij} \omega^i \wedge \omega^j , \\
d\theta^2 & \equiv -dp_2 \wedge \omega^1  + T^2_{ij} \omega^i \wedge \omega^j,
\end{aligned}
\end{equation}
where now the $T^k_{ij}$ are functions on $M'$ involving the known functions $f, g, h_1, h_2, k_i$ on $U$ and their derivatives, as well as the unknowns $u^1, u^2, p_2$ on $M'$.

As before, the next step is to determine whether there exist functions $p_{21}, p_{22}, p_{23}$ on $M$ such that the 1-form
\[ \pi_2 = dp_2 - p_{2j} \omega^j \]
satisfies
\begin{equation}\label{dthetas-second-pass-absorbed}
\begin{aligned}
d\theta^1 & \equiv  - \pi_2 \wedge \omega^2 , \\
d\theta^2 & \equiv -\pi_2 \wedge \omega^1.
\end{aligned}
\end{equation}
Similarly to the previous case, a straightforward (but rather involved) computation shows that such functions exist (and hence the torsion can be absorbed) if and only if
\begin{equation}\label{torsion-second-pass}
(2 k_3 - f) p_2 - Z_1 u^1
- Z_2 u^2
 = 0,
\end{equation}
where $Z_1, Z_2$ are long, complicated expressions involving the known functions $f, g, h_1, h_2, k_i$ on $U$ and their derivatives, and whose explicit form is not particularly enlightening.
At this point, there are (at least in principle) three possibilities to consider: 

\subsubsection{Case 3.1: Equation \eqref{torsion-second-pass} holds identically on $M'$.}  
This happens if and only if the coefficients of $p_2, u^1$, and $u^2$ each vanish identically on $U$.  This system of 3 PDEs is overdetermined, and differentiating yields several additional PDEs that arise as compatibility conditions, and which must be differentiated in turn to check for still more compatibility conditions.  Unfortunately, this process rapidly becomes computationally impractical to continue, and we have been unable to carry it to completion in order to determine definitively whether or not this PDE system admits solutions. However, we conjecture that it does not. 

Nevertheless, suppose that there exists some function $f$ for which this condition holds.  Then there is a unique integral element at each point of $M'$; however, the system $(M' \calI')$ is not involutive and must be prolonged.  
Because there exists a unique integral element at each point, this amounts to adding the 1-form
\[ \theta^3 = \pi_2 = dp_2 -  p_{21} \omega^1 - p_{22} \omega^2 - p_{23} \omega^3 \]
to $\calI'$, where $p_{21}, p_{22}, p_{23}$ are the (unique in this case) functions determined by the torsion absorption condition.  
Then we must compute $d\theta^3 \equiv 0$ modulo $\{\theta^1, \theta^2, \theta^3\}$ to see whether this condition introduces any additional constraints.
\begin{itemize}
\item If $d\theta^3 \equiv 0  \mod{\{\theta^1, \theta^2, \theta^3\}}$ identically on $M'$, then the system $\calI'^{(1)} = \{\theta^1, \theta^2, \theta^3\}$ is Frobenius of rank 3, and there is a 3-dimensional family of integral manifolds.  
We note that the condition $d\theta^3 \equiv 0  \mod{\{\theta^1, \theta^2, \theta^3\}}$ represents 3 additional PDEs that must be satisfied by the known functions on $U$.  Even if the PDE system given by the vanishing of equation \eqref{torsion-second-pass} were to be satisfied, it seems extremely unlikely that these additional PDEs and their compatibility conditions would be satisfied as well; however, we cannot rule out the possibility entirely.

\item If $d\theta^3 \not\equiv 0  \mod{\{\theta^1, \theta^2, \theta^3\}}$, then the equation $d\theta^3 \equiv 0  \mod{\{\theta^1, \theta^2, \theta^3\}}$ defines a (possibly empty) submanifold $M'' \subset M'$ to which the non-prolonged ideal $\calI'$ must be pulled back.   If this submanifold is not empty, then it is defined by a relation of the form
\[ Y_0 p_2 + Y_1 u^1 + Y_2 u^2 = 0. \]
(The fact that $M''$ is defined by a single equation is a consequence of some of the PDEs obtained by differentiating the equations given by the vanishing of \eqref{torsion-second-pass}.)
\begin{itemize}
\item If $Y_0 \neq 0$, then this equation can be solved for $p_2$, and the pullback of $\calI'$ to $M''$ is a rank 2 system $\calI''$ that is either Frobenius (which requires that additional PDEs be satisfied), in which case there is a 2-dimensional family of integral manifolds, or the conditions $d\theta^1 \equiv d\theta^2 \equiv 0 \mod{\{\theta^1, \theta^2\}}$ define one or two algebraic relations between $u^1$ and $u^2$, which must also be differentiated to check for additional compatibility conditions.  In the latter case, there is at most a 1-dimensional space of integral manifolds.
\item If $Y_0 = 0$ (which represents an additional PDE), then this equation defines an algebraic relation between $u^1$ and $u^2$, and any solution $(u^1, u^2)$ to the PDE system \eqref{first-order-PDE-sys} must have the form
\[ u^1 = -Y_2 u, \qquad u^2 = Y_1 u \]
for some function $u:U \to \R$.  Substituting these expressions into the PDE system \eqref{first-order-PDE-sys} yields a system of 4 first-order PDEs for the unknown function $u$.  This system has both algebraic and differential compatibility conditions that must be satisfied in order for any nonzero solutions to exist. If all these conditions are satisfied, then there is exactly a 1-dimensional space of integral manifolds.
\end{itemize}
In either case, there are further PDEs in addition to the PDE system given by the vanishing of equation \eqref{torsion-second-pass} that must be satisfied in order for integral manifolds to exist.  
Again, we consider it extremely unlikely that there exist any solutions of this form, but we cannot rule out the possibility entirely.  This is the basis for the second statement in Conjecture \ref{nonconstant-Beltrami-conjecture}.
\end{itemize}

\subsubsection{Case 3.2: The coefficient $(2 k_3 - f)$ of $p_2$ in \eqref{torsion-second-pass} vanishes identically on $U$, but at least one of the other two coefficients does not.}  
Then equation \eqref{torsion-second-pass} takes the form
\[ Z_1 u^1 + Z_2 u^2 = 0, \]
and the analysis is similar to that in the last bullet point above, with the result that there is at most a 1-dimensional family of integral manifolds.  We do not know of any examples satisfying this condition, but we note that this condition imposes fewer PDEs on the known functions on $U$ than the previous case, so there may be examples of this form.

\subsubsection{Case 3.3: The coefficient $(2 k_3 - f)$ of $p_2$ in \eqref{torsion-second-pass} is nonzero on $U$.}
In this case, equation \eqref{torsion-second-pass} defines a codimension 1 submanifold $M'' \subset M'$, diffeomorphic to $U \times \R^2$ (with coordinates $(u^1, u^2)$ on the $\R^2$ factor), with the property that any integral manifold of $(M, \calI)$ must be contained in $M''$.  Thus we must replace $\calI'$ with its pullback $\calI''$ to $M''$ and consider the system $(M'', \calI'')$.  This system is still generated by $\theta_1, \theta_2$ as in \eqref{define-I}, with $p_1$ and $p_2$ determined by equations \eqref{torsion-first-pass} and \eqref{torsion-second-pass}.
Modulo $\{\theta_1, \theta_2\}$, we now have
\begin{equation}\label{nonumbilic-dthetas-3}
\begin{aligned}
d\theta^1 & \equiv  T^1_{ij} \omega^i \wedge \omega^j , \\
d\theta^2 & \equiv T^2_{ij} \omega^i \wedge \omega^j,
\end{aligned}
\end{equation}
where now the $T^k_{ij}$ are functions on $M''$ involving the known functions $f, g, h_1, h_2, k_i$ on $U$ and their derivatives, as well as the unknowns $u^1, u^2$ on $M''$.

There is now nowhere to absorb the torsion functions $T^i_{jk}$ if they are nonzero.  There are two scenarios under which the system $(M'', \calI'')$ may have nontrivial integral manifolds:
\begin{itemize}
\item If $d\theta^1 \equiv d\theta^2 \equiv 0  \mod{\{\theta^1, \theta^2\}}$---i.e., if all the torsion functions $T^k_{ij}$ vanish identically on $M''$---then the system $(M'', \calI'')$ is Frobenius of rank 2, and there is a 2-dimensional family of integral manifolds. This is precisely what happens in Example \ref{cylinders-ex}, where the level sets of $f$ are concentric circular cylinders.

\item If $d\theta^1, d\theta^2$ are not both identically zero modulo $\{\theta^1, \theta^2\}$, then the equations $T^k_{ij} = 0$ define a (possibly empty) submanifold $M''' \subset M''$ to which $\calI''$ must be pulled back.
If this submanifold is nonempty, then it is defined by a single equation of the form 
\[ X_1 u^1 + X_2 u^2 = 0, \]
where $X_1$ and $X_2$ are expressions involving the known functions $f, g, h_1, h_2, k_i$ on $U$ and their derivatives.  The remaining analysis is then similar to that in Case 3.2, and there is at most a 1-dimensional family of integral manifolds.
\end{itemize}
This completes the proof of Theorem \ref{nonconstant-Beltrami-theorem}.

One important open question remains: Which of the possibilities above corresponds to the generic case, where we might expect to find a space of Beltrami fields parametrized by 3 functions of 2 variables?  The third statement in Conjecture \ref{nonconstant-Beltrami-conjecture} reflects our belief that the generic case is the final case above, i.e., the case where $2k_3 - f \neq 0$ and the system $(M'', \calI'')$ is not Frobenius.  This conjecture is based on the observation that this case imposes the fewest constraints on $f$ and its associated functions on $U$.  Moreover, as we shall see in the following section, it is consistent with the classification of Beltrami fields that possess either a translation symmetry or a rotation symmetry.

\section{Beltrami fields with symmetry}\label{symmetry-sec}

In this section, we consider the simpler problem of classifying Beltrami fields $\bfu$ that possess either a translation symmetry or a rotation symmetry.  In both cases, we are able to give a complete classification of local Beltrami fields with the corresponding symmetry.

\subsection{Beltrami fields with a translation symmetry}\label{translation-subsec}

Assume that $\bfu$ is a Beltrami field that admits a translation symmetry. Without loss of generality, we will assume that
\[ \frac{\partial \bfu}{\partial x^3} = 0. \] 
It follows that $\frac{\partial f}{\partial x^3} = 0$ as well, and so the level surfaces of $f$ are cylinders over curves in the $(x^1, x^2)$ plane, with rulings parallel to the $x^3$-axis.  Moreover, the vector field $\e_1 = \frac{\partial}{\partial x^3}$ is a principal direction to each level surface of $f$ at every point $\bfx\in U$.  Since this vector field is constant on $U$, we have
\[ \omega^1_2 = \omega^3_1 = 0. \]
The vector fields $\e_2, \e_3$ may be written as
\[ \e_2 = \cos(\phi) \frac{\partial}{\partial x^1} + \sin(\phi) \frac{\partial}{\partial x^2}, \qquad
\e_3 = -\sin(\phi) \frac{\partial}{\partial x^1} + \cos(\phi) \frac{\partial}{\partial x^2} \]
for some function $\phi(x^1, x^2)$.  (Note that this function is determined by $f$; specifically, it is determined by the condition that $\nabla f$ is parallel to $\e_3$.)
 The dual forms are
\[ \omega^1 = dx^3, \qquad \omega^2 = \cos(\phi)\, dx^1 + \sin(\phi) \,dx^2, \qquad \omega^3 = -\sin(\phi)\, dx^1 + \cos(\phi)\, dx^2, \]
and the remaining connection form is given by
\[ \omega^3_2 = d\phi = \phi_2 \omega^2 + \phi_3 \omega^3. \]
Thus we have
\[ h_{11} = g_1 = k_1 = k_2 = k_3 = 0, \qquad h_{22} = \phi_2, \qquad g_2 = \phi_3. \]

Since we now have $u^1_1 = u^2_1 = 0$ by assumption, the PDE system \eqref{first-order-PDE-sys} reduces to the system
\begin{equation}\label{first-order-PDE-sys-translation-invariant}
\begin{aligned}
u^1_1  = & \ u^2_1  = u^1_2  = 0, \\
u^2_2 & = - \phi_3 u^2, \\
u^1_3 & = f  u^2, \\
u^2_3 & = \phi_2 u^2 - f u^1.
\end{aligned}
\end{equation}
In particular, all the first-order derivatives of $u^1$ and $u^2$ are determined, and we have a total differential system for these two unknown functions.  So, let $M = U \times \R^2$, with coordinates $(u^1, u^2)$ on the $\R^2$ factor, and
let $\calI$ be the differential ideal on $M$ generated by the 1-forms
\begin{equation}
\begin{aligned}
\theta^1 & = du^1 - f u^2\,  \omega^3, \\
\theta^2 & = du^2 + \phi_3\,\omega^2 + (f u^1 - \phi_2 u^2)\,\omega^3,
\end{aligned} \label{define-I-translation-invariant}
\end{equation} 
and their exterior derivatives.  Direct computation shows that, modulo $\{\theta^1, \theta^2\}$, we have
\begin{equation*}
\begin{aligned}
d\theta^1 & \equiv 0 \\
d\theta^2 & = \left( 2 \phi_3 f u^1 - (\phi_{22} + \phi_{33}) u^2 \right) \omega^2 \wedge \omega^3.
\end{aligned}
\end{equation*}
Thus, the torsion absorption condition is
\begin{equation}\label{torsion-translation-invariant}
2 \phi_3 f u^1 - (\phi_{22} + \phi_{33}) u^2 = 0.
\end{equation}
There are two possibilities to consider.
\begin{enumerate}
\item Equation \eqref{torsion-translation-invariant} is satisfied identically on $M$, in which case the system $(M, \calI)$ is Frobenius and there is a 2-dimensional space of integral manifolds.  This condition means that the coefficients of $u^1$ and $u^2$ must both vanish identically on $U$, which is the case if and only if
\[ \phi_3 = \phi_{22} + \phi_{33} = 0. \]
The condition $\phi_3 = 0$ implies that
\[ d\e_3(\e_3) = \e_2 \omega^2_3(\e_3)  = 0, \]
and hence that the integral curves of the vector field $\e_3$ in the $(x^1, x^2)$ plane are straight lines.  Moreover, $\phi_3 = 0$ implies that $\phi_{33} = 0$, and so the second equation reduces to $\phi_{22} = 0$.  This equation implies that the rate of change $\phi_2$ of the angle $\phi$ is constant (as a function of arc length) along each level curve of $f$ in the $(x^1, x^2)$ plane.  

Together, these conditions imply that the level curves of $f$ in the $(x^1, x^2)$ plane are either concentric circles (if $\phi_2 \neq 0$) or
parallel lines (if $\phi_2=0$). Hence, the level surfaces of $f$ are either concentric circular cylinders or parallel planes.  In the former case, these are exactly the Beltrami fields of Example \ref{cylinders-ex}; in the latter case, this shows that the infinite-dimensional space of Beltrami fields in Example \ref{f=z-ex} contains precisely a 2-dimensional subspace of Beltrami fields that admit a translation symmetry.

\item If equation \eqref{torsion-translation-invariant} does not vanish identically on $U$, then it defines an algebraic relationship between $u^1$ and $u^2$.  Any solution to the PDE system \eqref{first-order-PDE-sys-translation-invariant} must have the form
\begin{equation}\label{translation-invariant-general-solution}
 u^1 = (\phi_{22} + \phi_{33}) u, \qquad u^2 = 2 \phi_3 f u 
\end{equation}
for some function $u: U \to \R$.  Note that neither $u^1$ nor $u^2$ may vanish identically on $U$, since the differential equations \eqref{first-order-PDE-sys-translation-invariant} would then imply that the other one vanishes as well.  Thus, by restricting $U$ if necessary, we may assume that both coefficients $\phi_3, \phi_{22} + \phi_{33}$ are nonzero on $U$.

Substituting the expressions \eqref{translation-invariant-general-solution} into the PDE system \eqref{first-order-PDE-sys-translation-invariant} yields four algebraic equations for the two nontrivial first partial derivatives $u_2, u_3$.  The algebraic compatibility conditions that must be satisfied in order for these equations to admit solutions are
\begin{equation}\label{compat-conds-translation-invariant}
\begin{gathered}
(\phi_{22} + \phi_{33})_2 = \frac{1}{\phi_3} (\phi_{23} - \phi_2^2)(\phi_{22} + \phi_{33}), \\
(\phi_{22} + \phi_{33})_3 = \left(\frac{e^{-g}}{f} - \phi_2\right) (\phi_{22} + \phi_{33}) + 2 \phi_3 f^2 + \frac{1}{2\phi_3}\left( \phi_{22}^2 + 4 \phi_{22} \phi_{33} + 3 \phi_{33}^2 + 4 \phi_3^2 f_2 \right).
\end{gathered}
\end{equation}
Remarkably, these two PDEs for $\phi$, together with the equation $f_2 = 0$ (which says that $\nabla f$ is parallel to $\e_3$), are compatible, and the space of functions $f$ for which the associated function $\phi$ satisfies these equations is locally parametrized by 3 functions of 1 variable.  Even more remarkably, for any such $f$, the corresponding total differential system for $u$ imposes no additional conditions, and so there exists a 1-dimensional space of integral manifolds.

\end{enumerate}

Thus we have the following classification result for Beltrami fields that possess a translation symmetry:

\begin{theorem}\label{translation-invariant-theorem}
The space of Beltrami fields that possess a translation symmetry is locally parametrized by 3 functions of 1 variable.  Moreover:
\begin{itemize}
\item The space of proportionality factors $f$ admitting a nonzero Beltrami field with a translation symmetry is locally parametrized by 3 functions of 1 variable.
\item Any such function $f$ admits exactly a 1-dimensional space of Beltrami fields with a translation symmetry unless the level surfaces of $f$ are concentric circular cylinders or parallel planes, in which case $f$ admits exactly a 2-dimensional space of Beltrami fields with a translation symmetry. 
\end{itemize}

\end{theorem}

It turns out that we can actually describe these Beltrami fields fairly explicitly by working with the original coordinate-based PDE system \eqref{Beltrami-sys-expanded}.  With the symmetry assumption, this system reduces to
\begin{equation}\label{Beltrami-sys-expanded-translation-invariant}
\begin{gathered}
-\frac{\partial u^3}{\partial x^2} = f u^1, \\
\frac{\partial u^3}{\partial x^1} = f u^2, \\
\frac{\partial u^1}{\partial x^2} - \frac{\partial u^2}{\partial x^1} = f u^3, \\
\frac{\partial u^1}{\partial x^1} + \frac{\partial u^2}{\partial x^2} = 0.
\end{gathered}
\end{equation}
The last equation in \eqref{Beltrami-sys-expanded-translation-invariant} implies that we must have 
\[ u^1 = -\frac{\partial H}{\partial x^2}, \qquad u^2 = \frac{\partial H}{\partial x^1} \]
for some function $H(x^1, x^2)$.
Substituting these expressions into the first two equations in \eqref{Beltrami-sys-expanded-translation-invariant} yields
\begin{equation}\label{translation-invariant-intermediate-step}
 -\frac{\partial u^3}{\partial x^2} = -f \frac{\partial H}{\partial x^2},
 \qquad 
 \frac{\partial u^3}{\partial x^1} = f \frac{\partial H}{\partial x^1}. 
\end{equation}
In particular, $\nabla u^3$ is parallel to $\nabla H$, which implies that
\[ u^3 = \Phi \circ H \]
for some function $\Phi:I \subset \R \to \R$.
Moreover, equations \eqref{translation-invariant-intermediate-step} imply that
\[ f = \Phi' \circ H. \]
Finally, the third equation in \eqref{Beltrami-sys-expanded-translation-invariant} implies that
\[ \Delta H = -f u^3 = - (\Phi' \circ H) (\Phi \circ H). \]

So the general Beltrami field with a translation symmetry can be constructed as follows:
\begin{enumerate}
\item Choose a function $\Phi: I \subset \R \to \R$.  
\item Let $H:U \subset \R^2 \to \R$ be a solution of the PDE
\begin{equation}\label{Phi-H-PDE}
 \Delta H = - (\Phi' \circ H) (\Phi \circ H). 
\end{equation}
\item Then the vector field
\[ \bfu = -\frac{\partial H}{\partial x^2} \frac{\partial}{\partial x^1} + \frac{\partial H}{\partial x^1} \frac{\partial}{\partial x^2} + (\Phi \circ H) \frac{\partial}{\partial x^3} \]
is a Beltrami field with proportionality factor 
\[ f = \Phi' \circ H. \]
\end{enumerate}
Note that this construction agrees with the function count given in Theorem \ref{translation-invariant-theorem}: The arbitrary function $\Phi$ represents 1 function of 1 variable, and for each function $\Phi$ the solution space of the PDE \eqref{Phi-H-PDE} for $H$ is locally parametrized by 2 functions of 1 variable, giving an overall solution space locally parametrized by a total of 3 functions of 1 variable.

\subsection{Beltrami fields with a rotation symmetry}\label{rotation-subsec}

Assume that $\bfu$ is a Beltrami field that admits a rotation symmetry.  
Without loss of generality, we will assume that the $z$-axis is the axis of symmetry; then this assumption means that the components of $\bfu$ with respect to the standard orthonormal cylindrical frame field are independent of the angle coordinate $\theta$.
It follows that $\frac{\partial f}{\partial \theta} = 0$ as well, and so the level surfaces of $f$ are surfaces of revolution about the $z$-axis. 
Moreover, the vector field $\e_1 = \frac{1}{r} \frac{\partial}{\partial \theta}$ is a principal direction to each level surface of $f$ at every point $\bfx\in U$.
The vector fields $\e_2, \e_3$ may be written as
\[ \e_2 = \cos(\phi) \frac{\partial}{\partial r} + \sin(\phi) \frac{\partial}{\partial z}, \qquad
\e_3 = -\sin(\phi) \frac{\partial}{\partial r} + \cos(\phi) \frac{\partial}{\partial z} \]
for some function $\phi(r,z)$, which is determined by the condition that $\nabla f$ is parallel to $\e_3$.
 The dual forms are
\[ \omega^1 = r\, d\theta, \qquad \omega^2 = \cos(\phi)\, dr + \sin(\phi) \,dz, \qquad \omega^3 = -\sin(\phi)\, dr + \cos(\phi)\, dz, \]
and the connection forms are
\[ \omega^1_2 = \frac{\cos(\phi)}{r}\omega^1, \qquad \omega^3_1 = \frac{\sin(\phi)}{r}\omega^1, \qquad \omega^3_2 = d\phi = \phi_2\, \omega^2 + \phi_3\, \omega^3. \]
Thus we have
\[ g_1 = k_2 = k_3 = 0, \qquad h_{11} = \frac{\sin(\phi)}{r}, \qquad k_1 = \frac{\cos(\phi)}{r}, \qquad
 h_{22} = \phi_2, \qquad g_2 = \phi_3. \]

Since we now have $u^1_1 = u^2_1 = 0$ by assumption, the PDE system \eqref{first-order-PDE-sys} reduces to the system
\begin{equation}\label{first-order-PDE-sys-rotation-invariant}
\begin{aligned}
u^1_1 & = \ u^2_1 = 0, \\
u^1_2 & = -\frac{\cos(\phi)}{r} u^1, \\
u^2_2 & = - \left(\frac{\cos(\phi)}{r} + \phi_3\right) u^2, \\
u^1_3 & =  \frac{\sin(\phi)}{r} u^1 +  f  u^2, \\
u^2_3 & = \phi_2 u^2 - f u^1.
\end{aligned}
\end{equation}
The analysis of the corresponding exterior differential system is similar to that in the previous subsection, with the following results:
\begin{enumerate}
\item If there is a 2-dimensional space of integral manifolds, then
\[ \phi_3 = r^2 \phi_{22} - r \cos (\phi)\, \phi_2 + \cos(\phi) \sin(\phi) = 0. \]
Differentiating the second equation in the $\e_3$ direction (and taking into account that $\phi_3=0$) yields
\[ (2r\sin(\phi) - 3r^2 \phi_2) \phi_{22} + r\cos(\phi) \phi_2^2 - \cos(\phi) \sin(\phi) \phi_2 = 0. \]
Taken together, these equations imply that either $\cos(\phi)=0$, or
\begin{equation}\label{rot-inv-2d}
 \phi_3 = \phi_{22} = 0, \qquad \phi_2 = \frac{\sin(\phi)}{r}. 
\end{equation}
In the latter case, differentiating the last equation in \eqref{rot-inv-2d} in the $\e_2$ direction (and taking into account that $\phi_{22} = 0$) yields
\[ \cos(\phi) \sin(\phi) = 0. \]
Thus the only solutions are $\phi = \frac{\pi}{2}$, in which case the level surfaces of $f$ are concentric circular cylinders, or $\phi = 0$, in which case the level surfaces of $f$ are parallel planes.  In the former case, these are exactly the Beltrami fields of Example \ref{cylinders-ex}; in the latter case, this shows that the infinite-dimensional space of Beltrami fields in Example \ref{f=z-ex} contains precisely a 2-dimensional subspace of Beltrami fields that admit a rotation symmetry.

\item The space of functions $f$ admitting a 1-dimensional space of integral manifolds is locally parametrized by 3 functions of 1 variable.

\end{enumerate}

Thus we have the following classification result for Beltrami fields that possess a rotation symmetry:

\begin{theorem}\label{rotation-invariant-theorem}
The space of Beltrami fields that possess a rotation symmetry is locally parametrized by 3 functions of 1 variable.  Moreover:
\begin{itemize}
\item The space of proportionality factors $f$ admitting a nonzero Beltrami field with a rotation symmetry is locally parametrized by 3 functions of 1 variable.
\item Any such function $f$ admits exactly a 1-dimensional space of Beltrami fields with a rotation symmetry unless the level surfaces of $f$ are concentric circular cylinders or parallel planes, in which case $f$ admits exactly a 2-dimensional space of Beltrami fields with a rotation symmetry.
\end{itemize}

\end{theorem}

Again, it turns out that we can actually describe these Beltrami fields fairly explicitly by working with the original PDE system \eqref{Beltrami-def} in cylindrical coordinates.  If we write $\bfu$ in terms of the orthonormal cylindrical frame field as
\[ \bfu = u^1 \frac{\partial}{\partial r} + u^2 \frac{1}{r} \frac{\partial}{\partial \theta} + u^3 \frac{\partial}{\partial z}, \]
then the system \eqref{Beltrami-def} is equivalent to the PDE system 
\begin{equation}\label{Beltrami-sys-cylindrical}
\begin{gathered}
\frac{\partial u^3}{\partial \theta} - \frac{\partial (ru^2)}{\partial z} = f (r u^1), \\
\frac{\partial u^1}{\partial z} - \frac{\partial u^3}{\partial r} = f u^2, \\
\frac{\partial (r u^2)}{\partial r} - \frac{\partial u^1}{\partial \theta} = f (ru^3), \\
\frac{\partial (ru^1)}{\partial r} + \frac{\partial u^2}{\partial \theta} + \frac{\partial (ru^3)}{\partial z} = 0.
\end{gathered}
\end{equation}
With the symmetry assumption, this system reduces to
\begin{equation}\label{Beltrami-sys-expanded-rotation-invariant}
\begin{gathered}
 -\frac{\partial (ru^2)}{\partial z} = f (r u^1), \\
\frac{\partial u^1}{\partial z} - \frac{\partial u^3}{\partial r} = f u^2, \\
\frac{\partial (r u^2)}{\partial r} = f (ru^3), \\
\frac{\partial (ru^1)}{\partial r}  + \frac{\partial (ru^3)}{\partial z} = 0.
\end{gathered}
\end{equation}
The last equation in \eqref{Beltrami-sys-expanded-rotation-invariant} implies that we must have 
\[ ru^1 = -\frac{\partial H}{\partial z}, \qquad ru^3 = \frac{\partial H}{\partial r} \]
for some function $H(r,z)$.
Substituting these expressions into the first and third equations in \eqref{Beltrami-sys-expanded-rotation-invariant} yields
\begin{equation}\label{rotation-invariant-intermediate-step}
 -\frac{\partial (ru^2)}{\partial z} = -f \frac{\partial H}{\partial z},
 \qquad 
 \frac{\partial (ru^2)}{\partial r} = f \frac{\partial H}{\partial r}. 
\end{equation}
In particular, $\nabla (ru^2)$ is parallel to $\nabla H$, which implies that
\[ ru^2 = \Phi \circ H \]
for some function $\Phi:I \subset \R \to \R$.
Moreover, equations \eqref{rotation-invariant-intermediate-step} imply that
\[ f = \Phi' \circ H. \]
Finally, the second equation in \eqref{Beltrami-sys-expanded-rotation-invariant} implies that
\[ \frac{\partial}{\partial r} \left( \frac{1}{r} \frac{\partial H}{\partial r}\right) + \frac{\partial}{\partial z} \left( \frac{1}{r} \frac{\partial H}{\partial z}\right) = -f u^2 = -\frac{1}{r} (\Phi' \circ H) (\Phi \circ H). \]

So the general Beltrami field with a rotation symmetry can be constructed as follows:
\begin{enumerate}
\item Choose a function $\Phi: I \subset \R \to \R$.  
\item Let $H:U \subset \R^2 \to \R$ be a solution of the PDE
\begin{equation}\label{Phi-H-PDE-rot}
 \frac{\partial}{\partial r} \left( \frac{1}{r} \frac{\partial H}{\partial r}\right) + \frac{\partial}{\partial z} \left( \frac{1}{r} \frac{\partial H}{\partial z}\right) =  -\frac{1}{r} (\Phi' \circ H) (\Phi \circ H). 
\end{equation}
\item Then the vector field
\[ \bfu = -\frac{1}{r} \frac{\partial H}{\partial z} \frac{\partial}{\partial r} + \frac{1}{r^2}(\Phi \circ H) \frac{\partial}{\partial \theta} + \frac{1}{r} \frac{\partial H}{\partial r} \frac{\partial}{\partial z} \]
is a Beltrami field with proportionality factor 
\[ f = \Phi' \circ H. \]
\end{enumerate}
As in the translation symmetry case, this construction agrees with the function count given in Theorem \ref{rotation-invariant-theorem}: The arbitrary function $\Phi$ represents 1 function of 1 variable, and for each function $\Phi$ the solution space of the PDE \eqref{Phi-H-PDE-rot} for $H$ is locally parametrized by 2 functions of 1 variable, giving an overall solution space locally parametrized by a total of 3 functions of 1 variable.

\bibliographystyle{amsplain}
\bibliography{Beltrami-bib}

\end{document}